\documentclass[letterpaper, 10 pt, conference]{ieeeconf}
\makeatletter
\let\proof\@undefined                   
\let\endproof\@undefined                
\makeatother

\usepackage{amsmath}
\usepackage{amsfonts}
\usepackage{amsthm}
\usepackage{graphicx}
\usepackage{epstopdf}
\usepackage{framed}

\newtheorem{thm}{Theorem}
\newtheorem{lem}{Lemma}
\newtheorem{cor}{Corollary}

\title{From Global Linear Computations to Local Interaction Rules}
\author{Zak Costello and Magnus Egerstedt\\
School of Electrical and Computer Engineering\\
Georgia Institute of Technology\\
\texttt{\{zak.costello,magnus\}@gatech.edu}}
\date{}

\begin{document}
\maketitle

\begin{abstract}
A network of locally interacting agents can be thought of as performing a distributed computation. But not all computations can be faithfully distributed. This paper investigates
 which global, linear transformations can be computed using local rules, i.e., rules which rely solely on information from adjacent nodes in a network. The main result states that a linear transformation is computable in finite time using local rules if and only if the transformation has positive determinant.
An optimal control problem is solved for finding the local interaction rules, and simulations are performed to elucidate how optimal solutions can be obtained.
\end{abstract}

\section{Introduction}
One common theme when designing control and coordination mechanisms for distributed, multi-agent systems is that the information, on which decisions are based, is restricted to be shared among agents that are adjacent in the underlying information-exchange network, e.g., \cite{BCM09,mesbahi2010graph,olfati2007consensus,RenBeard}. As a result, local rules are needed for processing the information and coordinating the agents in the network in such a way that some global objective is achieved.
Problems that fit this description can be found in a variety of applications, including power systems \cite{dimeas2005operation,Prosumer,SG}, formation control \cite{balch1998behavior,JE07,Jadbabaie03,Tanner03,vijay}, distributed sensor networks \cite{romer2004design,fumin}, smart textiles \cite{marculescu2003electronic}, and distributed optimization \cite{cortes2005coordination,NOP10}. In this paper we take initial steps towards developing a general theory of local implementability/computability of such global behaviors. 

As such, one key aspect of
algorithm design is the definition of local interaction rules that produce desired global behaviors. An example of this are consensus algorithms for computing averages in a distributed manner. In fact, consensus plays a role in many different applications, including multi-agent robotics, distributed sensor fusion, and power network control, e.g., \cite{olfati2007consensus,Prosumer,olfati2005consensus}. To this end, let the scalar state of each node in a network be $x_i \in \mathbb{R}$, with initial condition $x_{i}(t_0)=\xi_i, i = 1,\dots,n$, where $n$ is the number of nodes in the network.  By stacking the states together in $x\in\mathbb{R}^n$, we implicitly perform an asymptotic, global computation through the
so-called consensus equation 
\begin{equation}
\dot{x}_i = -\sum\limits_{j \in N_i} (x_i - x_j),
\end{equation}
where $N_i$ encodes a neighborhood relationship in the underlying information-exchange network. And, as long as the network is connected and undirected, all node values will converge to the same value, namely the average of the initial conditions, e.g., \cite{mesbahi2010graph}. In other words, 
\begin{equation}
\lim\limits_{t \to \infty} x(t)= \frac{1}{n}\begin{bmatrix}
1 & \ldots & 1 \\
\vdots & \ddots & \vdots \\
1 & \ldots & 1
\end{bmatrix}\xi,
\end{equation}
where $\xi$ is the vector containing all the initial node values. As such, the 
consensus equation is asymptotically computing the average, which is
a global property since it relies on the state of every node in the network. 

In this work, we are interested in problems where networks are tasked with computing arbitrary linear transformations of the initial node states.
 In particular, we answer two fundamental questions: {\it What global, linear transformations can be computed using local rules? How do we find the local rules that would compute a given linear transformation?} This paper answers both questions and presents existence guarantees of local rules for a given computation. 


Some work has been done in the general area of obtaining global information with local interactions. In \cite{sundaram2008distributed}, a fixed weighting scheme was used to compute linear transformations on networks. That work focused on cases where each agent computes the same linear transformation, as is the case with the consensus computation where each node computes the average, while we, in this paper, do not wish to restrict ourselves to this special case.
 In a certain sense, the investigation in
\cite{rotkowitz2006characterization} follows this line of inquiry as well. There, quadratic invariance was used to establish whether or not a convex optimization problem
exists whose solution is a decentralized implementation of a centralized
 feedback controller. \cite{swigart2009graph} further expounds on this idea and provides a practical, graph theoretic method for finding this distributed controller. Our work distinguishes itself from this body of work by using a time varying weighting method, which admits the computation of global, linear transformations in finite time. 

In fact,
in this paper, we consider computations that are to be performed using local rules over a static and undirected information-exchange network. The local rules, once obtained, admits a decentralized implementation, where ``decentralized" in this context means that each node in the network only needs to communicate state information among adjacent nodes in the network. In particular, we ask if it is possible to define local interaction laws such that $x(t_f)=T\xi$, given the linear transformation $T$ and the initial conditions $x(t_0)=\xi$. Necessary and sufficient conditions are given for this to be possible, and they state that local interaction rules exist if and only if $T$ has positive determinant.

The remainder of this paper is organized as follows:  In Section \ref{ProbDef}, the problem 
under consideration is introduced
 and the general class of admissible, local interaction rules is established. In Section \ref{Cntrl}, necessary and sufficient conditions are presented under which 
global, linear computations can be performed in a decentralized manner.
In Section \ref{OptimalControlProb}, an optimal control problem is formulated, which provides a way to find 
interaction rules,
the time varying weighting functions numerically, and in Section \ref{Applications} two 
instantiations
of the method are presented 
together with simulation results.

\section{Problem Definition}
\label{ProbDef}
To formalize what is meant by local interactions, we first need to discuss the information-exchange network over which the interactions are defined. To this end, let $V$ be a vertex set with cardinality $n$, and $E \subset V \times V$ be an edge set with cardinality $m$, where we insist on $(i,i)\in E,~\forall i\in V$, as well as $(i,j)\in E\Leftrightarrow(j,i)\in E$. Let $G$ be the graph $G = (V,E)$, where the assumptions on $E$ imply that $G$ is undirected and contains self-loops. We moreover assume that $G$ is connected. As the main purpose with $G$ is to encode adjacency information in the information-exchange network, we introduce the operator $\mbox{sparse}(G)$ to capture these adjacencies, and we say that an $n\times n$ matrix $M\in \mbox{sparse}(G)$ if $(i,j)\not\in E\Rightarrow M_{ij}=0$.

There are a number of different ways in which local interactions can be defined. In this paper, we assume that they are given by time-varying, piecewise continuous weights associated with the edges in the network. If $x_i\in\mathbb{R}$ is the scalar state associated with node $i\in V$, we define a local interaction as a continuous-time process
\begin{equation}
\dot x_i(t)=\sum_{j|(i,j)\in E}w_{ij}(t)x_j(t).
\end{equation}
Note that we do not insist on $w_{ij}=w_{ji}$ even though $G$ is undirected,  as shown in Figure \ref{fig:graph}.

\begin{center}
\begin{figure}[thpb]
\begin{center}
\includegraphics[scale=0.6]{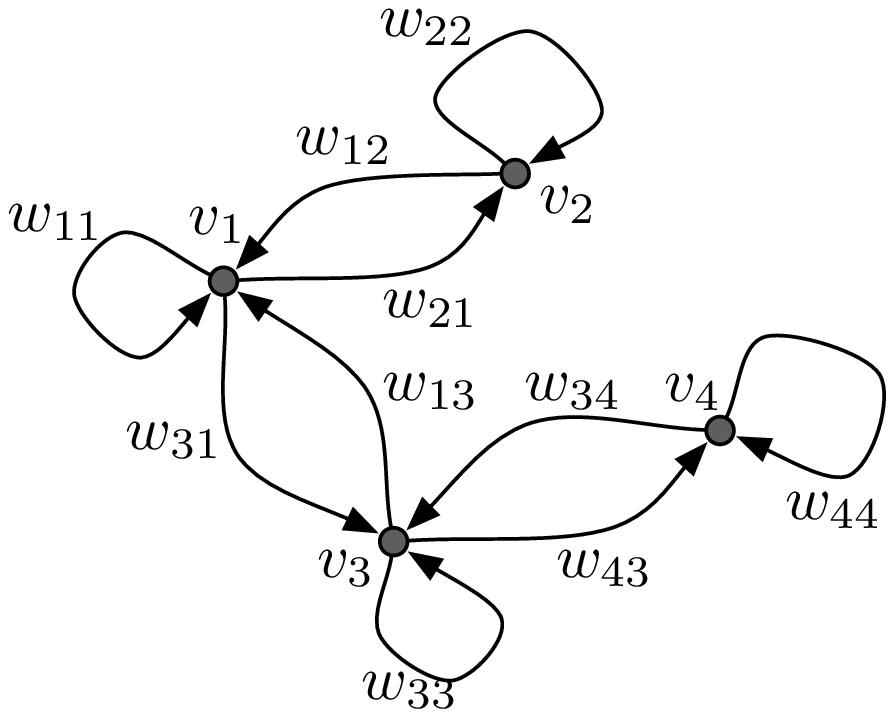}
\end{center}
\begin{equation*}
W = \begin{bmatrix}
w_{11} & w_{12} & w_{13} & 0 \\
w_{21} & w_{22} & 0 & 0 \\
w_{31} & 0 & w_{33} & w_{34} \\
0 & 0 & w_{43} & w_{44} \\
\end{bmatrix}\in sparse(G)
\end{equation*}
\caption{An example of the sparsity structure and the local interaction rules used in this paper.}
      \label{fig:graph}
\end{figure}
\end{center}

If we stack the states together in $x=[x_1,\ldots,x_n]^T\in\mathbb{R}^n$, what we mean by {\it local interactions} is thus
\begin{equation}
\dot x(t)=W(t)x(t),~~W(t)\in\mbox{sparse}(G),
\label{eq:sparse}
\end{equation}
with solution
\begin{equation}
x(t)=\Phi(t,t_0)x(t_0),
\label{eq:sol}
\end{equation}
where $\Phi$ is the state transition matrix associated with the system in Equation \ref{eq:sparse}, e.g., \cite{RWB}.

.

Now, the purpose of the local interactions is to perform a global, linear computation. In other words, given the $n\times n$ matrix $T$ and the initial condition $x(t_0)=\xi$, what we would like to do is find $W(t)\in\mbox{sparse}(G),~t\in[t_0,t_f]$, such that 
\begin{equation}
x(t_f)=T\xi.
\end{equation}
But, comparing this expression to Equation \ref{eq:sol}, this simply means that what we would like is
\begin{equation}
\Phi(t_f,t_0)=T.
\end{equation}
If this was indeed the case, then the local interactions, as defined through $W(t)$, would indeed compute $T\xi$ over the interval $[t_0,t_f]$ for all possible values of $\xi$, i.e., 
one can think of the network as a black box that takes $\xi$ as the input at time $t_0$ and, at time $t_f$, returns $T\xi$ as the output.

As a final observation before we can formulate the general problem of performing global, linear computations using local interactions, we note that state transition matrix satisfies the same dynamics as 
Equation \ref{eq:sparse}, i.e.,
\begin{equation}
\label{expandedDynamics}
\frac{d{\Phi}(t,t_0)}{dt} = W(t)\Phi(t,t_0),
\end{equation}
with initial condition $\Phi(t_0,t_0)=I$, where $I$ is the $n\times n$ identity matrix.


\medskip
\noindent
{\bf Problem 1 [Local Computation]}

\noindent
{\it 
Given a linear transformation $T$ and a connected graph $G$,
find $W(t)\in \mbox{sparse}(G),~t\in[t_0,t_f]$, such that
\begin{equation}
\dot{\mathbf{X}}(t) = W(t)\mathbf{X}(t),
\end{equation}
with boundary conditions $\mathbf{X}(t_0)=I,~\mathbf{X}(t_f)=T$.}

%

\section{On the Existence of Solutions}
\label{Cntrl}
The main point with this paper is an exploration of what linear transformations $T$ admit a local implementation, i.e., for what $T$ Problem 1 has a solution. In this section, we develop necessary and sufficient conditions for this to be the case.

We start by observing that since $\mathbf{X}(t)$ is really the state transition matrix $\Phi(t,t_0)$, it is always invertible,
\begin{equation}
\mathbf{X}(t)^{-1}=\Phi(t,t_0)^{-1}=\Phi(t_0,t).
\end{equation}
As a direct consequence of this, $T$ has to be invertible for a solution to Problem 1 to exist, i.e., we need that $\mbox{det}(T)\neq0$. But, as $\mathbf{X}(0)=I$, we have that $\mbox{det}(\mathbf{X}(0))=1>0$. Moreover, the determinant of a matrix depends continuously on its entries, and therefore the only way for $\mbox{det}(\mathbf{X}(\tau))<0$ for some $\tau\in(t_0,t_f]$, there has to exist a $\tau'\in(t_0,\tau)$ such that $\mbox{det}(\mathbf{X}(\tau'))=0$. But this can not happen since $\mathbf{X}$ is always invertible. From this it directly follows that for Problem 1 to have a solution, $T$ has to satisfy $\mbox{det}(T)>0$. 

To state this fact more compactly, let $\mbox{GL}_{+}^n(\mathbb{R})$ denote the set of all $n\times n$, real matrices with positive determinant. We have thus established the following necessary condition for the existence of a solution:

%
%

\begin{lem} 
\label{nosol}
A solution to Problem 1 exists only if $T\in\mbox{GL}_{+}^n(\mathbb{R})$.
\end{lem}

%


One consequence of Lemma  \ref{nosol} is that it is impossible to use local rules, as understood in this paper, to achieve consensus in finite time. This follows directly from the fact that the consensus computation is given by the linear map
\begin{equation}
T_{cons}=\frac{1}{n}\mathbf{1}^T\mathbf{1},
\label{eq:tcons}
\end{equation}
where $\mathbf{1}$ is a vector of length $n$, with all entries equal to one. And,
$$
\mbox{rank}(T_{cons})=1,
$$
i.e., $\mbox{det}(T_{cons})=0$. We state this fact as a corollary:

\begin{cor}
There is no solution to Problem 1 which admits finite time consensus.\begin{footnote}{Note that this applies to any agreement across the nodes, i.e., not only to average consensus.}\end{footnote}
\end{cor}
%

Now that we have established necessary conditions for Problem 1 to have a solution, we turn our attention to sufficient conditions. And, surprisingly enough, $T\in\mbox{GL}_{+}^n(\mathbb{R})$ turns out to be both necessary and sufficient for a solution to exist, which constitutes the main result in this paper:

\begin{framed}
\begin{thm}
\label{necAndSuff}
A solution to Problem 1 exists if and only if $T\in\mbox{GL}_{+}^n(\mathbb{R})$.
\end{thm}
\end{framed}

As we have already established sufficiency, what must be shown is that whenever $\mbox{det}(T)>0$, there is a $W(t)\in\mbox{sparse}(G)$ that drives $\mathbf{X}$ from $I$ to $T$.
The remainder of this section is devoted to the establishment of this fact. However, before we can give the proof to Theorem \ref{necAndSuff}, a number of supporting results are needed, involving the controllability of nonlinear, drift-free  systems, i.e., systems of the form
\begin{equation}
\label{driftFree}
\dot{x} = \sum\limits_{i=1}^p g_i(x) u_i,
\end{equation}
where $x\in\mathbb{R}^n$ is the state of the system, and $u_1,\ldots,u_p\in\mathbb{R}$ are the control inputs. For the sake of easy reference, we start by recalling Chow's Theorem, as formulated in \cite{sastry1999nonlinear}, for such drift-free systems:


\begin{thm}[Chow's Theorem, e.g. \cite{sastry1999nonlinear}]
The system in Equation \ref{driftFree} is locally controllable 
about a point $x_0$ if and only if
\begin{equation}
dim(\overline{\Delta}(x_0)) = n,
\end{equation}
where $\overline{\Delta}$ is the involutive closure of the distribution $\mbox{span}\{g_1,\ldots,g_p\}$.
\end{thm}
The system is moreover controllable if it is locally controllable everywhere. And, the proof that $T\in\mbox{GL}_{+}^n(\mathbb{R})$ is sufficient for Problem 1 to have a solution will hinge on showing that the dynamics, as defined through the local interaction rules in Equation \ref{eq:sparse}, is indeed controllable everywhere on $\mbox{GL}_{+}^n(\mathbb{R})$. To this end, we first must rewrite the dynamics in Problem 1 on the appropriate form. For this, 
we need the index matrix $\mathbb{I}_{ij} \in \mathbb{R}^{n\times n}$, which has a one at the $i$th row and $j$th column, and zeros everywhere else.
The index matrix allows us to rewrite
$$
\dot{\mathbf{X}}=W\mathbf{X}
$$
as
\begin{equation}
\dot{\mathbf{X}} = \left( \sum\limits_{i = 1}^n \sum\limits_{j = 1}^n W \odot \mathbb{I}_{ij} \right) \mathbf{X},
\end{equation}
 where the $\odot$ symbol represents element-wise matrix product, i.e.,
\begin{equation}
\dot{\mathbf{X}} = \left(
\begin{bmatrix}
w_{11} & \ldots & 0 	\\
\vdots       & \ddots & \vdots \\
0            & \ldots & 0  	\\
\end{bmatrix} + \ldots + 
\begin{bmatrix}
0 	&  \ldots  & 0 		\\
\vdots	& \ddots     & \vdots 	\\
0       & \ldots  & w_{nn}  	\\
\end{bmatrix}
\right) \mathbf{X},
\end{equation}
where we have surpressed the explicit dependence on $t$ for the sake of notational ease.

Rearranging the terms and letting
\begin{equation}
\label{gDef}
g_{ij}(\mathbf{X}) = \mathbb{I}_{ij} \mathbf{X},
\end{equation}
we get the drift-free matrix formulation
\begin{equation}
\label{driftFreeSystem}
\dot{\mathbf{X}} = \sum\limits_{i=1}^n \sum\limits_{j|(i,j)\in E} g_{ij}(\mathbf{X}) w_{ij}.
\end{equation}

To clarify, $g_{ij}(\mathbf{X})$ is a matrix whose $i$th row contains the $j$th row of $\mathbf{X}$, with the rest of the elements in the matrix equal to $0$,
\begin{equation}
g_{ij} (\mathbf{X})= 
\begin{matrix}
1 \\
\vdots \\
i - 1 \\
i \\
i + 1 \\
\vdots \\
n
\end{matrix}
\begin{bmatrix}
0 & \ldots & 0 \\
\vdots & \ddots & \vdots \\
0 & \ldots & 0 \\
\mathbf{X}_{j1} & \ldots & \mathbf{X}_{jn} \\
0 & \ldots & 0 \\
\vdots & \ddots & \vdots \\
0 & \ldots & 0 \\
\end{bmatrix}.
\end{equation}

As a final step towards a formulation that is amenable to Chow's Theorem, let the
vectorized version of $g_{ij}$ be given by 
$\vec{g}_{ij} = \mbox{vec}(g_{ij})$,
resulting in the
vectorized version of Equation \ref{driftFreeSystem},
\begin{equation}
\mbox{vec}(\dot{\mathbf{X}}) = \sum\limits_{i=1}^n \sum\limits_{j|(i,j)\in E} \vec{g}_{ij}(\mathbf{X}) w_{ij}.
\label{eq:vec}
\end{equation}

The first order of business towards establishing controllability of this system is the derivation of the Lie brackets for the system in Equation \ref{eq:vec}.
\begin{lem}
\label{liebracketlemma}
\begin{equation}
[\vec{g}_{ij}(\mathbf{X}),\vec{g}_{kl}(\mathbf{X})] = 
\begin{cases}
-\vec{g}_{il}(\mathbf{X})  & \text{if } j = k,\text{ }i \neq l\\
 \vec{g}_{kj}(\mathbf{X}) & \text{if } i = l,\text{ }j \neq k\\
\mathbf{0} & \text{otherwise}
\end{cases}
\end{equation}
\end{lem}

\begin{proof}
The Lie bracket $[\vec{g}_{ij},\vec{g}_{kl}]$ is given by
\begin{equation}
\frac{\partial \vec{g}_{kl}}{\partial \mbox{vec}(\mathbf{X})} \vec{g}_{ij} - \frac{\partial \vec{g}_{ij}}{\partial \mbox{vec}(\mathbf{X})} \vec{g}_{kl},
\label{eq:Lie}
\end{equation}
where we have suppressed the explicit dependence on $\mathbf{X}$. 

Substitution of Equation \ref{gDef} into Equation \ref{eq:Lie}, the above expression yields
\begin{equation}
\label{lem1eqn1}
\frac{\partial (\mbox{vec}(\mathbb{I}_{kl}\mathbf{X}))}{\partial \mbox{vec}(\mathbf{X})} \mbox{vec}(\mathbb{I}_{ij}\mathbf{X}) - \frac{\partial (\mbox{vec}(\mathbb{I}_{ij}\mathbf{X}))}{\partial \mbox{vec}(\mathbf{X})} \mbox{vec}(\mathbb{I}_{kl}\mathbf{X}),
\end{equation} 
which can be rewritten, using the Kronecker product, as
\begin{equation}
\begin{aligned}
\frac{\partial((I \otimes \mathbb{I}_{kl}) \mbox{vec}(\mathbf{X}))}{\partial \mbox{vec}(\mathbf{X})} (I \otimes \mathbb{I}_{ij}) \mbox{vec}(\mathbf{X}) \\- \frac{\partial ((I \otimes \mathbb{I}_{ij}) \mbox{vec}(\mathbf{X}))}{\partial \mbox{vec}(\mathbf{X})} (I \otimes \mathbb{I}_{kl})\mbox{vec}(\mathbf{X})
\end{aligned}
\end{equation}
Taking the derivatives yields
\begin{equation}
(I \otimes \mathbb{I}_{kl})(I \otimes \mathbb{I}_{ij}) \mbox{vec}(\mathbf{X}) -  (I \otimes \mathbb{I}_{ij})(I \otimes \mathbb{I}_{kl})\mbox{vec}(\mathbf{X}).
\end{equation}
Using the mixed product property of the Kronecker product, Equation \ref{eq:Lie} can be further simplified as
\begin{equation}
(I \otimes \mathbb{I}_{kl} \mathbb{I}_{ij}) \mbox{vec}(\mathbf{X}) -  (I \otimes \mathbb{I}_{ij}\mathbb{I}_{kl})\mbox{vec}(\mathbf{X}),
\end{equation}
i.e., the Lie bracket in Equation \ref{eq:Lie} becomes
\begin{equation}
\label{lem1eqn2}
[\vec{g}_{ij}(\mathbf{X}),\vec{g}_{kl}(\mathbf{X})] = 
\mbox{vec}(\mathbb{I}_{kl} \mathbb{I}_{ij} \mathbf{X}) - \mbox{vec}(\mathbb{I}_{ij}\mathbb{I}_{kl} \mathbf{X}).
\end{equation}

Now,
using the fact that, $\mathbb{I}_{ij} \mathbb{I}_{kl} = \mathbb{I}_{il}$ if $j = k$ and $\mathbb{I}_{ij} \mathbb{I}_{kl} = 0$ otherwise, we can break down Equation \ref{lem1eqn2} into 3 cases:  First if $j = k$ and $i \neq j$ we get
$
[\vec{g}_{ij},\vec{g}_{kl}]=
- \mbox{vec}(\mathbb{I}_{il} \mathbf{X}) = -\vec{g}_{il}$.
The second case occurs when $i = l$ and $j \neq k$, in which case
$
[\vec{g}_{ij},\vec{g}_{kl}]=
\mbox{vec}(\mathbb{I}_{kj} \mathbf{X}) = \vec{g}_{kj}$.
Otherwise, the Lie bracket is $\mathbf{0}$, and the lemma follows. 
\end{proof}


Now that Lie brackets can be computed in general for this problem, we must determine if the involutive closure of the distribution associated with the system in Equation \ref{eq:vec} contains enough independent vector fields for local controllability. To help with this determination, we provide the following lemma.

\begin{lem}
\label{connectionLemma}
If node $i$ is path-connected to node $j$, then $\vec{g}_{ij}(\mathbf{X})$ is in the distribution $\overline{\Delta}(\mathbf{X})$.
\end{lem}

\begin{proof}
That node $i$ is path-connected to node $j$ means that there is a path through adjacent nodes in the graph $G$ that starts at node $i$ and ends at node $j$.  Assume that the path goes through the nodes
$N_1, \ldots, N_q$, i.e., $N_1$ is adjacent to $N_2$, $N_2$ is adjacent to $N_3$, and so forth,
while $N_1 = i$ and $N_q = j$. Since these nodes are adjacent, we, by definition, have that
$\vec{g}_{N_1N_2}, \vec{g}_{N_2N_3}, \ldots, \vec{g}_{N_{q-1}N_q} \in \Delta(\mathbf{X})$.

The involutive closure contains every possible Lie bracket that can be recursively created from elements $\Delta(\mathbf{X})$, which implies that the problem is to create $\vec{g}_{ij}$ from some combination of Lie brackets from elements in $\Delta(\mathbf{X})$. 
And, from Lemma \ref{liebracketlemma}, we know that
$[\vec{g}_{N_1N_2},\vec{g}_{N_2N_3}]$ is equal to $-\vec{g}_{N_1N_3}$. Applying Lemma \ref{liebracketlemma} again gives
$[-\vec{g}_{N_1N_3},\vec{g}_{N_3N_4}] = \vec{g}_{N_1N_4}$. This procedure can be repeated until we arrive at one of two possible cases. If $q$ is even, the result is $[-\vec{g}_{N_1N_{q-1}},\vec{g}_{N_{q-1}N_q}] = \vec{g}_{N_1N_q}$. If $q$ is odd we get $[\vec{g}_{N_1N_{q-1}},\vec{g}_{N_{q-1}N_q}] = -\vec{g}_{N_1N_q}$. In either case, we are able to construct $\vec{g}_{N_1N_q}$ from previous Lie brackets, as shown in Figure \ref{fig:LB}. And, as $N_1 = i$ and $N_q = j$, we have $\vec{g}_{ij}\in \overline{\Delta}(\mathbf{X})$. 
\end{proof}

\begin{center}
\begin{figure}[thpb]
\begin{center}
\includegraphics[scale=0.6]{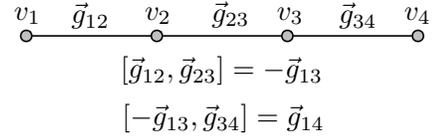}
\end{center}
\caption{An example of the construction in the proof of
Lemma \ref{connectionLemma} with node $i$ and $j$ being represented by $v_1$ and $v_4$, respectively.  
}
      \label{fig:LB}
\end{figure}
\end{center}

To establish that the system is controllable on $\mbox{GL}_{+}^n(\mathbb{R})$, $\overline{\Delta}(\mathbf{X})$ must have rank $n^2$ everywhere on this set, which is the topic of the next lemma.
\begin{lem}\label{bigThm} If $G$ is connected then 
 $\overline{\Delta}(\mathbf{X})$ has dimension $n^2$ if and only if $\mbox{rank}(\mathbf{X}) = n$. 
\end{lem}

\begin{proof}To prove this lemma, we need to show that the implication goes both ways. 

Assume first that $\mbox{dim}(\overline{\Delta}(\mathbf{X})) = n^2$.  If $G$ is connected then, by Lemma \ref{connectionLemma}, 
\begin{equation}
\overline{\Delta}(\mathbf{X})=\mbox{span}\{\vec{g}_{ij},~\forall (i,j) \in V\times V\}.
\end{equation}
Since $V\times V$ has cardinality $n^2$, we can conclude that 
each $\vec{g}_{ij}$ is in $\overline{\Delta}(\mathbf{X})$. 

For the purpose of the proof, it is convinient to go back to the matrix formulation, and we recall that
 $\vec{g}_{ij} = vec(g_{ij})$. 
As such, we will use
the matrix form $g_{ij}$ 
to construct $\mathbf{X}$.  And, since the goal is to form a matrix with rank $n$, only $n$ linearly independent matrices are needed.  So, we arbitrarily choose to form $\mathbf{X}$ from the ``diagonal" set $\{g_{11}, g_{22}, \ldots, g_{nn} \}$.  Using the fact that $g_{ij} = \mathbb{I}_{ij} \mathbf{X}$, we can write,
\begin{equation*}
\sum\limits_{i=1}^n g_{ii} = \sum\limits_{i=1}^n \mathbb{I}_{ii} \mathbf{X},
\end{equation*}
which simplifies to
\begin{equation}
\sum\limits_{i=1}^n g_{ii} = \mathbf{X}.
\end{equation}
And, since $\mathbf{X}$ is a linear combination of $n$ linearly independent matrices, $\mbox{rank}(\mathbf{X})=n$, and the first implication follows.



Next, we must show
that
\begin{equation}
\mbox{rank}(\mathbf{X}) = n \Rightarrow  \mbox{dim}(\overline{\Delta}(\mathbf{X})) = n^2,
\end{equation}
which we do by contradiction.  Using the expression $g_{ij} = \mathbb{I}_{ij} \mathbf{X}$, $n^2$ matrices can be formed from $\mathbf{X}$.  Let us assume that they are not linearly independent.  This implies that there exists a set of coefficients $\alpha_{ij}$ such that, for some $(k,l)$,
\begin{equation}
\sum\limits_{(i,j) \neq (k,l)} g_{ij} \alpha_{ij} = g_{kl}.
\label{eq:XX}
\end{equation}
Since $\mathbf{X}$ has full rank, 
$\mathbf{X}$ can be removed from Equation \ref{eq:XX} based on the fact that $g_{ij} = \mathbb{I}_{ij} \mathbf{X}$, yielding
\begin{equation}
\sum\limits_{(i,j) \neq (k,l)} \mathbb{I}_{ij} \alpha_{ij} = \mathbb{I}_{kl}.
\label{eq:false}
\end{equation}

By definition of the index matrix, Equation \ref{eq:false} cannot be true, since
every matrix in the sum on the left has a value of zero where $\mathbb{I}_{kl}$ has value of $1$. Therefore, we have reached a contradiction and can conclude that $\mbox{dim}(\overline{\Delta}(\mathbf{X}) = n^2$.


\end{proof}

Since $\mathbf{X}$ is really a state transition matrix, i.e., it is indeed invertible (with $\mbox{rank}(\mathbf{X})=n$),
the system in Equation \ref{driftFreeSystem} is locally controllable 
everywhere on $\mbox{GL}^n_{+}(\mathbb{R})$ as long as the underlying graph $G$ is connected:
\begin{thm}
The system 
$$
\dot{\mathbf{X}}=W\mathbf{X},~W\in\mbox{sparse}(G)
$$
is locally controllable everywhere on $\mbox{GL}^n_{+}(\mathbb{R})$ if $G$ is connected.
\label{cor:ctrb}
\end{thm}

Theorem \ref{cor:ctrb} and Lemma \ref{nosol} give us all the ammunition needed to prove the main result in this paper, namely Theorem \ref{necAndSuff}:
%
%
%
\begin{proof}[Proof of Theorem \ref{necAndSuff}]
Lemma \ref{nosol} tells us that a solution only exists if $T\in\mbox{GL}_{+}^n(\mathbb{R})$, so what remains is to establish that this is indeed sufficient. Hence, assume that $T\in\mbox{GL}_{+}^n(\mathbb{R})$. Since $I\in\mbox{GL}_{+}^n(\mathbb{R})$, and $\mbox{GL}_{+}^n(\mathbb{R})$ is connected \cite{strang}, there is a continuous curve of matrices in $\mbox{GL}_{+}^n(\mathbb{R})$ that connects $I$ and $T$. And, by Corollary \ref{cor:ctrb}, every point along the path connecting $I$ and $T$ is locally controllable. The system being drift-free moreover implies that it can flow along this curve, e.g., \cite{brockett1972system}.
Therefore, a solution to Problem 1 exists if $T\in\mbox{GL}_{+}^n(\mathbb{R})$.
\end{proof}


If we return to the conensus problem, we have already established that $T_{cons}$ in Equation \ref{eq:tcons} is not computable in finite time using local rules. However, consider instead the transformation
\begin{equation}
T_{cons2}=\left[\begin{array}{cccc}
1/n&1/n&\cdots&1/n\\
0&1&\cdots&0\\
\vdots&\vdots&\ddots&\vdots\\
0&0&\cdots&1\end{array}\right].
\end{equation}
We have
\begin{equation}
\mbox{det}(T_{cons2})=\frac{1}{n}
\end{equation}
and, as such, it is computable using local rules. In this case, the network average is only computed by a single node (node 1 in this case), while the remaining nodes return to their initial values at the end of the computation. This can in fact be generalized to any scalar, non-zero, linear map $\ell:\mathbb{R}^n\rightarrow\mathbb{R}$ through
$$
T_\ell\xi=\left[\begin{array}{c}
\ell(\xi)\\
\xi_2\\
\vdots\\
\xi_n
\end{array}\right],
$$
where we have assumed that $\ell(\xi)$ depends on $\xi_1$.\begin{footnote}{If not, simply pick another node in the network that $\xi$ does depend on, as the node where the computation takes place.}\end{footnote}
The point with this is that {\it it is possible to compute any scalar, non-zero, linear map as long as the computation only has to take place at a single node.}

%
%

\section{Optimal Local Interactions}
\label{OptimalControlProb}
Just because we know that a computation $T\xi$ can be done using local rules it does not follow that we can (easily) find these rules, encoded through $W(t)\in\mbox{sparse}(G)$, such that $\dot{\mathbf{X}}=W\mathbf{X},~\mathbf{X}(t_0)=I, \mathbf{X}(t_f)=T$. In this section, we address this problem in the context of optimal control.
%
%
%

Let the cost be given by
\begin{equation}
J(W) = \int\limits_0^{t_f} \frac{1}{2} \| W(t) \|_F^2 dt,
\end{equation}
where $\|\cdot\|_F$ is the Frobenius norm. The resulting constrained minimization problem becomes

\medskip
\noindent
{\bf Problem 2 [Optimal Local Interactions]}

\noindent
{\it 
\begin{equation}
\min_{W}  J(W) = \int\limits_0^{t_f} \frac{1}{2} \| W(t) \|_F^2 dt
\end{equation}
such that
\begin{equation}
\begin{array}{l}
\dot{\mathbf{X}} = W\mathbf{X}\\ 
W(t)\in\mbox{sparse}(G),~\forall t\in[t_0,t_f]\\
\mathbf{X}(t_0) = I,~\mathbf{X}(t_f)=T.
\end{array}
\end{equation}
}
%
%
%
%
The Hamiltonian associated with Problem 2 (e.g., \cite{liberzon2012calculus}), with costate matrix $\lambda$, is given by
\begin{equation}
H = \mbox{vec}(\lambda)^T  \mbox{vec}( W\mathbf{X}) + \frac{1}{2} \| W\|^2_F.
\end{equation}
We can rewrite the Hamiltonian as
\begin{equation}
H = \sum\limits_{i=1}^n  \sum\limits_{j |(i,j)\in E}\sum\limits_{k=1}^n \lambda_{ik} w_{ij} \mathbf{X}_{jk} + \frac{1}{2} \sum\limits_{i=1}^n \sum\limits_{j |(i,j)\in E} w_{ij}^2.
\end{equation}
The optimality conditions are
\begin{equation}
0 = \frac{\partial H}{\partial w_{ij}} = \sum\limits_{k=1}^n \lambda_{ik} \mathbf{X}_{jk}   + w_{ij},
\end{equation}
i.e., the optimal weights are given by
\begin{equation}
\label{optimalityCondition}
w_{ij} =  -\sum\limits_{k=1}^n \lambda_{ik}\mathbf{X}_{jk},
\end{equation}
which  yields $m + n$ optimality conditions. This is also the number of nonzero values in the $W$ matrix. 

We get  the costate equations from the derivative of the Hamiltonian with respect to $\mathbf{X}$:
\begin{equation}
\dot{\lambda}_{ij} = - \frac{\partial H}{\partial \mathbf{X}_{ij}} = - \sum\limits_{k | (i,k)\in E}  w_{ki} \lambda_{kj}.
\end{equation}

By substituting the optimality conditions into both the state and costate equations, we get $2n$ equations with initial and final conditions on $\mathbf{X}_{ij}$. The resulting, two-point boundary problem becomes
\begin{eqnarray}
\dot{\mathbf{X}}_{ij} &=& -\sum\limits_{k|(i,k)\in E} \mathbf{X}_{kj} \sum\limits_{l = 1}^n \lambda_{il}\mathbf{X}_{kl}\nonumber\\
\label{eqn:twoPointBoundryProblem}
\mathbf{X}(t_0) &=& I,~\mathbf{X}(t_f)~=~T\\
\dot{\lambda}_{ij} &=& \sum\limits_{k |(i,k)\in E} \lambda_{kj} \sum\limits_{l = 1}^n \lambda_{kl}\mathbf{X}_{il},\nonumber
\end{eqnarray}
which can be 
solved numerically, as will be seen in the next section.



\section{Simulations and Examples}
\label{Applications}
Computing linear transformation of states can be useful in a variety of network applications. In this section we explore two concrete examples. The first involves improving the convergence rates in distributed computations and the second involes information exchange among non-local agents.

\subsection{Improving Convergence Rates}
Consider, again, the consensus equation,
\begin{equation}
\dot{x} = -L_sx,
\label{eq:Ls}
\end{equation}
where
$x \in \mathbb{R}^n$ is the state of the system and $L_s$ is the graph Laplacian associated with a given, {\it sparse} yet connected graph $G_s$. 
These dynamics are known to
converge exponentially in the algebraic connectivity of the graph $G_s$, e.g.,\cite{olfati2007consensus}, where the algebraic connectivity is given by the second smallest eigenvalue of $L_s$. 

For some applications this convergence rate may not be fast enough.  Instead, one might want to make the system behave as if the graph was more dense; thus improve the convergence rate. Let $G_d$ be a {\it dense} graph, with corresponding consensus dynamics
\begin{equation}
\dot{z} = -L_dz,
\label{eq:Ld}
\end{equation}
where $z$ represents a ``desired" state (how we would like $x$ to behave) and $L_d$ is the graph Laplacian of  $G_d$.  The desired system in Equation \ref{eq:Ld} has state transition matrix $\Phi_d(t,t_0)$, with $\Phi_d(t_0,t_0)=I$, and
\begin{equation}
\Phi_d(t_f,t_0) = e^{-L_d(t_f-t_0)},
\end{equation}
which we thus set as the linear transformation we would like to compute by the original, sparse network. 

In other words, let
$T = e^{-L_d(t_f-t_0)}$.  The goal is to compute $T$ using local time varying weights on the graph $G_s$. And, since matrix exponentials are invertible, $T\in\mbox{GL}_{+}^n(\mathbb{R})$, and a solution does indeed exist.


As an instantiation of this, consider a 5 node system. We can solve Problem 2 numerically, using test-shooting, for this system in order to find
$W$ such that
\begin{equation}
\begin{array}{l}
\dot{\mathbf{X}} = W\mathbf{X},~W\in\mbox{sparse}(G_s)\\
\mathbf{X}(t_0)=I,~\mathbf{X}(t_f) = e^{-L_d(t_f-t_0)}.
\end{array}
\end{equation}
The solution to this problem yields a set of time varying weights $W$ and matrices $\mathbf{X}$.
The weights can be executed in a decentralized manner, once they have been obtained (using centralized computations). The result is a
sparse network that acts at a higher rate, as if it was indeed dense.
This is shown in Figures \ref{StateFunctions} - \ref{fig:cons}. In Figure \ref{fig:cons}, the agreement error,
$$
\|x(t)-{1}/{n}\mathbf{1}^T\mathbf{1}x(0)\|
$$
is shown both for the original (sparse) consensus dynamics (Equation \ref{eq:Ls}) and for the optimal, ``densified" version.

\begin{center}
\begin{figure}[thpb]
\includegraphics[scale=0.66]{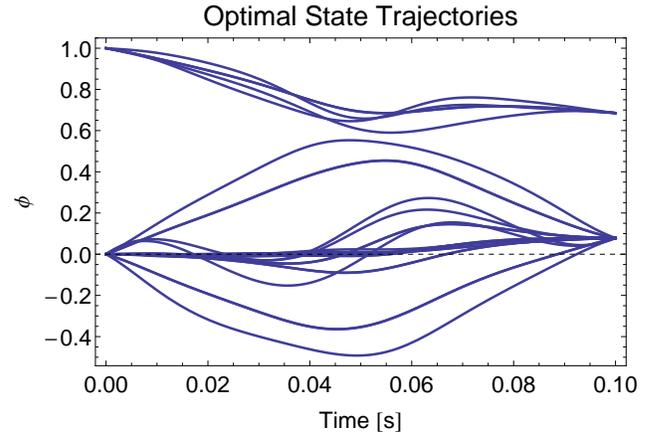}
\caption{This plot shows the evolution of each element of the state transition matrix $\Phi(t,t_0)$ over the time interval from $t_0$ to $t_f$. Since $\Phi(t_0,t_0)=I$ all elements are either $0$ or $1$ initially. 
}
      \label{StateFunctions}
\end{figure}
\end{center}

\begin{figure}[thpb]
\centering
\includegraphics[scale=0.66]{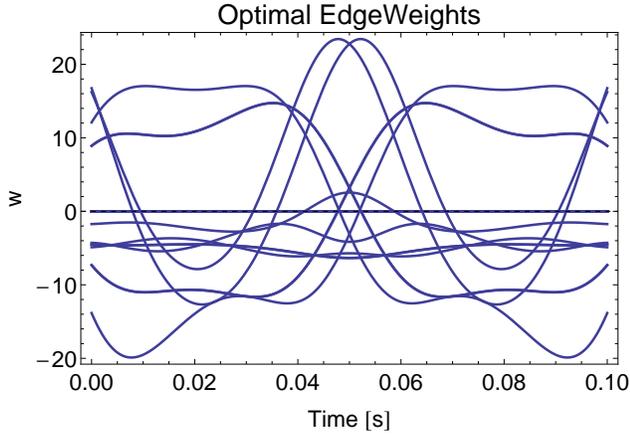}
\caption{The local weight functions each agent uses in order to compute the linear transformation given by $e^{-L_d(t_f-t_0)}$.  The weights are computed by solving Problem 2 numerically.}
\label{WeightFunctions}
\end{figure}

\begin{center}
\begin{figure}[thpb]
\begin{center}
\includegraphics[scale=0.47]{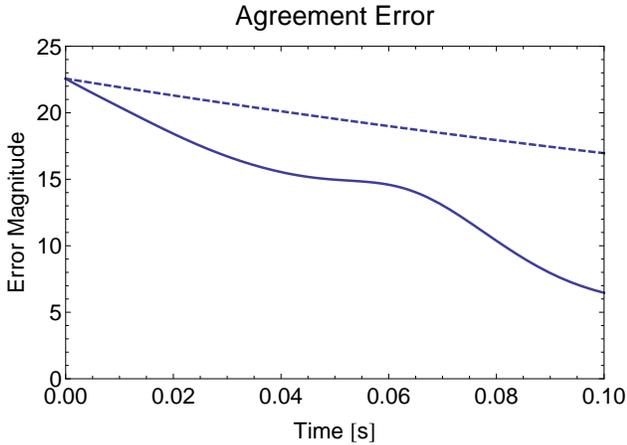}
\end{center}
\caption{The agreement errors for the original, sparse consensus equation (dashed) and the ``densified" version (solid). As expected, the latter has a higher rate of convergence.}
      \label{fig:cons}
\end{figure}
\end{center}

\subsection{Swapping Node Values}
As another example, consider the situation when 
the linear transformation represents a reordering (or swapping) of states. For a $4$ node case, where agents 1 and 2 and agents 3 and 4 are to ``swap" state values, the transformation matrix becomes
\begin{equation}
T_{swap} = \begin{bmatrix}
0 & 1 & 0 & 0 \\
1 & 0 & 0 & 0 \\
0 & 0 & 0 & 1 \\
0 & 0 & 1 & 0
\end{bmatrix}.
\end{equation}
However, the linear interpolation between $I$ and $T_{swap}$ contains a singular matrix, which makes the two-point boundary problem numerically ill-conditioned when using shooting methods, e.g., \cite{shoot}. A way around this problem is to avoid this singular matrix by solving two sequential two-point boundary problems.
%
%
%
%


As an example, in the first iteration, we let the boundary conditions be $\mathbf{X}(t_0)=I,~\mathbf{X}((t_f-t_0)/2)=T_1$. For the second iteration, they are $\mathbf{X}((t_f-t_0)/2)=T_1,~\mathbf{X}(t_f)=T_{swap}$, where
\begin{equation}
\begin{aligned}
T_1 = 
\begin{bmatrix}
0 & 1 & 0 & 0 \\
0 & 0 & 1 & 0 \\
1 & 0 & 0 & 0 \\
0 & 0 & 0 & 1 \\
\end{bmatrix}.
\end{aligned}
\end{equation}

This sequential approach avoids the numerical ill-conditioning, and the solution is shown in Figures \ref{WeightFunctions2} - \ref{WeightFunctions4}.

\begin{figure}[thpb]
\centering
\includegraphics[scale=0.47]{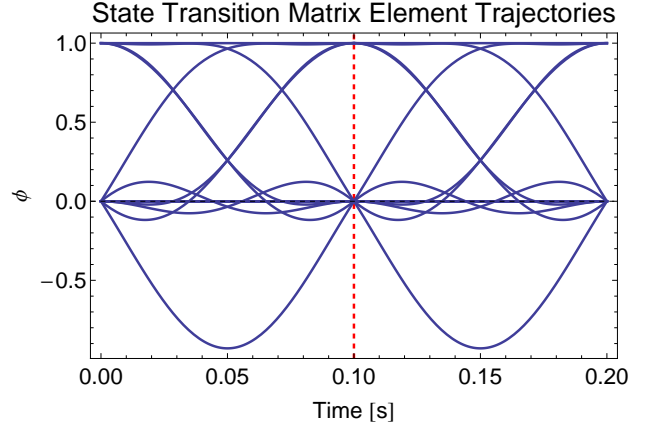}
\caption{The evolution of the state transition matrix for the 4-node swapping problem, with $\Phi(t_0,t_0)=I,~\Phi((t_f-t_0)/2,t_0)=T_1$, and $\Phi(t_f,t_0)=T_{swap}$.}
\label{WeightFunctions2}
\end{figure}

\begin{figure}[thpb]
\centering
\includegraphics[scale=0.47]{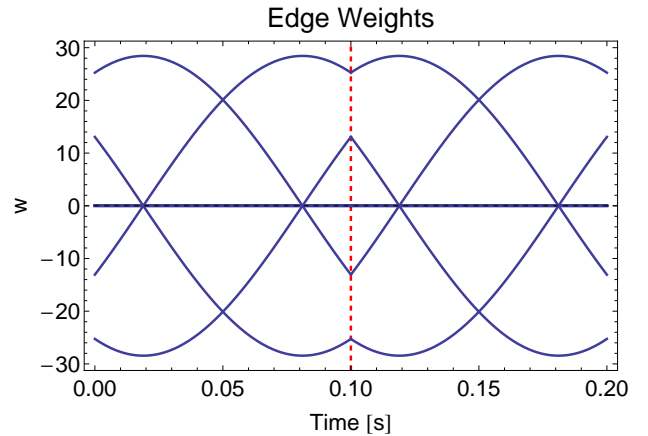}
\caption{The weight functions define the local interactions needed to achieve the swap in the 4-node case.}
\label{WeightFunctions3}
\end{figure}

\begin{figure}[thpb]
\centering
\includegraphics[scale=0.47]{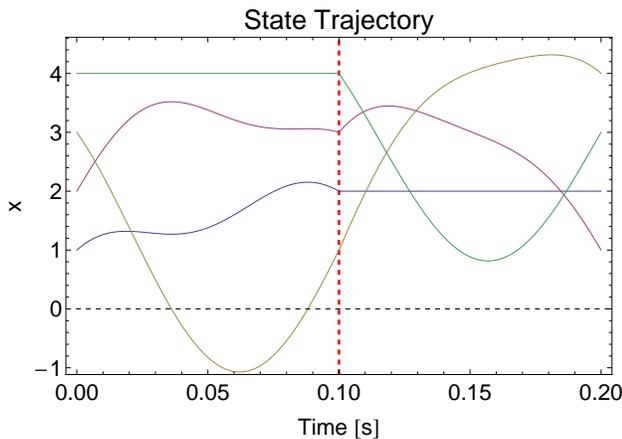}
\caption{The evolution of the node states for the swap problem. The initial state is $x(t_0) = [1, 2, 3, 4]^T$ and the final state is $x(t_f) = [2,1,4,3]^T$, i.e., the first and second states swapped values and the third and fourth state swapped values.}
\label{WeightFunctions4}
\end{figure}

\section{Conclusions}
\label{conclusion}


In this paper, a step was taken
towards computing arbitrary global functions on networks with local interaction rules. 
In particular, it presented a method 
which allows a networked system to compute global, linear transformations using only local rules. 

We derived necessary and sufficient conditions under which it is possible to use
a distributed, time-varying weighting scheme to compute the transformation $T$ for undirected, connected networks with fixed topology. Specifically, we showed that 
the necessary and sufficient condition for $T$ to be locally computable is that it has positive determinant, i.e., $T\in\mbox{GL}_{+}^n(\mathbb{R})$.


\section*{Acknowledgment}
\addcontentsline{toc}{section}{Acknowledgment}
This work was sponsored in part by a grant from the US Air Force Office for Sponsored Research.
The authors would like to thank Professor
Mark Costello at the Georgia Institute of Technology for his advice regarding this work. 

\bibliographystyle{IEEEtran}
\bibliography{references}
\end{document}